\documentclass[12pt, a4paper]{article}
\usepackage{custarticle}

\title{Special K\"ahler structures, cubic differentials and hyperbolic metrics}
\author{Andriy Haydys and Bin Xu}
\date{10th March 2020}

\begin{document}
\maketitle

\begin{abstract}
  We obtain necessary conditions for the existence of special K\"ahler structures with isolated singularities on compact Riemann surfaces.
  We prove that these conditions are also sufficient in the case of the Riemann sphere and, moreover, we determine the whole moduli space of special K\"ahler structures  with fixed singularities.
  The tool we develop for this aim is a correspondence between special K\"ahler structures and pairs consisting of a cubic differential and a hyperbolic metric.
 \end{abstract}


\section{Introduction}

For reader's convenience, let us recall the definition of the affine special K\"ahler structure, which is the main object of study of this preprint.

\begin{definition}[\cite{Freed99_SpecialKaehler}]
	\label{Defn_spK}
	An \emph{(affine) special K\"ahler structure} on a manifold $\Sigma$ is a quadruple $(g,I,\omega, \nabla)$, where $(\Sigma,g,I,\omega)$ is a K\"ahler manifold with Riemannian metric $g$, complex structure $I$, and symplectic form $\omega(\cdot,\cdot)=g(I\cdot,\cdot)$, and $\nabla$ is a flat symplectic torsion-free connection on the tangent bundle $T\Sigma$ such that
	\begin{equation}
	\label{Eq_SpKaehlerCondition}
	(\nabla^{}_XI)Y=\, (\nabla^{}_Y I)X
	\end{equation}
	holds for all vector fields $X$ and $Y$.
\end{definition}

If $I$ is fixed, which is always assumed to be the case below, we say for simplicity that $(g,\nabla)$ is a special K\"ahler structure.

The notion of a special K\"ahler structure has its origin in physics~\cites{Gates84_SuperspaceFormulation_SpKaehler, deWitVanProyen84PotentialsSymmetries} and is the natural structure of the base of an algebraic integrable system~\cite{Freed99_SpecialKaehler}.
In particular, algebraic integrable systems appear naturally in gauge theory~\cites{Hitchin:87, SeibergWitten:94, SeibergWittenMonopoles:94,Nekrasov03_SWprepotential, Gaiotto12_N2Dualities}, where a special instance of an algebraic integrable system---the Seiberg--Witten curve---plays a central  r\^ole in the (physical) Seiberg--Witten theory.
Very recently, special K\"ahler structures on Riemann surfaces have been extensively studied from the perspective of $\cN =2$ superconformal field theory, see~\cite{ArgyreslueMartone17_SWgeomCoulombBranch} and references therein.

Examples of special K\"ahler structures can be found in~\cite{Freed99_SpecialKaehler, Haydys15_IsolSing_CMP, CalliesHaydys17_LocalModels_IMRN_online, GTZ:16, OO:18, CVZ:19}.
An elementary introduction to special K\"ahler geometry on Riemann surfaces can be found in~\cite{CalliesHaydys17_AffSpK2D}.

Note that a \emph{complete} special K\"ahler metric is necessarily flat~\cite{Lu99_NoteOnSpKaehler, BauesCortes01_SpKParabolicSpheres}.
Besides, singularities of fibers of an algebraic integrable system lead to singularities of the corresponding special K\"ahler structure.
This motivates studies of \emph{singular} special K\"ahler metrics as a natural  structure on bases of algebraic integrable systems.

\medskip

Associated to a special K\"ahler structure is the period map $\tau$, which takes values in the Siegel upper half-space~\cite{Freed99_SpecialKaehler}.
If $\Sigma$ is a Riemann surface, which is assumed to be the case below unless otherwise stated explicitly, $\tau$ takes values in the upper half-plane $\Sieg:=\{ z\in \C\mid \Im z>0 \}$, which is endowed with the standard hyperbolic metric $g_{\Sieg}=\bigl (\Im z\bigr)^{-2}|dz|^2$.

\begin{defn}
	\label{Defn_AssocHypMetr}
	Let $\Sigma$ be a Riemann surface.
	For a special K\"ahler structure $(g,\nabla)$ on  $\Sigma$ with the period map $\tau$ we call $\tilde g:=\tau^*g_{\Sieg}$ \emph{the associated hyperbolic metric}.
\end{defn}

Notice that $\tau$ depends on certain choices and, moreover, is defined locally only (or, equivalently, on the universal covering of $\Sigma$), however  $\tilde g$ is well-defined.
Also, $\tilde g$ may either degenerate or be singular at isolated points, hence, strictly speaking, $\tilde g$ is a metric outside of some discrete subset of $\Sigma$.
This is not a concern for us, since we are interested in singular special K\"ahler structures, which involves singular metrics anyway.

In Definition~\ref{Defn_AssocHypMetr} and below, a metric is said to be \emph{hyperbolic}, if its Gaussian curvature is constant and equals $-1$.
If $\tilde g$ is any hyperbolic metric on $\Sigma$, we say that $\tilde g$ \emph{represents a divisor
$\Sigma_{j=1}^n\, (\alpha_j-1)\,p_j$ with $0\leq \alpha_j\not=1$} if the following holds: If $\alpha_j =0$, then $\tilde g$ has a {\it cusp singularity} at $p_j$; If $\alpha_j>0$, $\tilde g$ has a {\it conical singularity of order} $\alpha_j-1$, i.e. $\tilde g$ has a conical angle of $2\pi\alpha_j$ at $p_j$ (See the precise explanation in Corollary \ref{Thm_SpKHyperbLoc}).

Recall~\cite{Freed99_SpecialKaehler} that for any special K\"ahler structure we can also construct  \emph{the associated cubic form} $\Xi$, which is a holomorphic section of $K_\Sigma^3$, where $K_\Sigma$ is the canonical bundle of $\Sigma$.
Throughout this manuscript we assume that $\Xi$ is non-zero.
This means that we exclude special K\"ahler structures $(g, \nabla)$, where $g$ is flat and $\nabla$ is the Levi--Civita connection of $g$.

Thus, to any special K\"ahler structure on a Riemann surface we can associate a pair $(\tilde g,\Xi)$ as above.
Our main result, \autoref{Thm_spKviaHyperbXi_short} below, states, roughly speaking, that for any pair $(\tilde g,\Xi)$ consisting of a hyperbolic metric possibly singular at isolated points and a meromorphic cubic form we can construct a special K\"ahler structure, whose associated hyperbolic metric and associated cubic form are $\tilde g$ and $\Xi$ respectively.

The precise statement requires some notations, which are introduced next.
Denote by $(r, \theta)$ the polar coordinates on $\C^*$, where $\theta\in (0,2\pi)$, and  put $\rho:=\log r$,
\[
\mathbbm 1_2:=
\begin{pmatrix}
1 & 0\\
0 & 1
\end{pmatrix},\qquad\text{and}\qquad
\rI_2:=
\begin{pmatrix}
0 & -1\\
1 & \phantom{-}0
\end{pmatrix}.
\]
For any $k\in\Z$ and $b\in\C\setminus\{ 0\}$ the following
\begin{equation}
\label{Eq_ModelLogSing}
\begin{aligned}
g_{k } &= -|b|r^{k}\log r\, |dz|^2,\\
\omega_{k,\nabla} &=
\frac 12
\left(k\rI_2 +
\begin{pmatrix}
\Im\!\left(\tfrac{b}{|b|} e^{ik\theta}\right) & -1+\Re\!\left(\tfrac{b}{|b|} e^{ik\theta}\right)\\
1 + \Re\!\left(\tfrac{b}{|b|} e^{ik\theta}\right) &  -\Im\!\left(\tfrac{b}{|b|} e^{ik\theta}\right)
\end{pmatrix} \rho^{-1} \right) d\theta  \\
&+
\frac 12\left(k\mathbbm 1_2	+
\begin{pmatrix}
1 - \Re\!\left(\tfrac{b}{|b|} e^{ik\theta}\right) & \Im\!\left(\tfrac{b}{|b|} e^{ik\theta}\right)\\
\Im\!\left(\tfrac{b}{|b|} e^{ik\theta}\right) & 1+\Re\!\left(\tfrac{b}{|b|} e^{ik\theta}\right)
\end{pmatrix}\rho^{-1}
\right) d\rho
\end{aligned}
\end{equation}
is a special K\"ahler structure on the punctured disc $\{0<r<1\}$~\cite{CalliesHaydys17_LocalModels_IMRN_online}.
Here $\om_{k,\nabla}$ is the connection one-form of $\nabla$ \wrt the trivialization $(\partial_x, \partial_y)$, where $z=x+ yi$ is the standard coordinate on $\C$.
By~\cite{CalliesHaydys17_LocalModels_IMRN_online}*{Thm.\,5}, \eqref{Eq_ModelLogSing} together with flat cones
\begin{equation}
\label{Eq_FlatConicalStr}
g^c_\b = r^{\b} |dz|^2,\qquad \omega^c_{\b,\nabla}=\omega_{LC} = \frac \b 2
\bigl ( \mathbbm 1_2\, d\rho + \rI_2\, d\theta \bigr),
\end{equation}
where $\b\in\R$, are local models of isolated singularities of affine special K\"ahler structures in complex dimension one provided the associated cubic form is meromorphic.

\begin{defn}\cite{CalliesHaydys17_LocalModels_IMRN_online}*{Def.\,6}
	We say that a special K\"ahler structure $(g,\nabla)$ on the punctured disc has \emph{a conical singularity} of order $\tfrac 12\b$ at the origin, if $(g,\om_\nabla)$ is asymptotic to $(g^c_\b, \om^c_{\b,\nabla})$.
	We say that $(g,\nabla)$ has \emph{a logarithmic singularity} of order $\tfrac 12 k,\ k\in \Z$, at the origin, if  $(g,\om_\nabla)$ is asymptotic to $(g_{k}, \om_{k,\nabla})$.
\end{defn}

\begin{remark}
	Geometrically, \eqref{Eq_FlatConicalStr} can be thought of as follows:
	\begin{itemize}[itemsep=-2pt,topsep=2pt]
		\item If $\beta >-2$, \eqref{Eq_FlatConicalStr}  is a cone of total angle $\pi (\beta + 2)$.
		\item If $\beta = -2$,  \eqref{Eq_FlatConicalStr} is a cylindrical end with the origin at infinity.
		\item If $\beta < -2$,  \eqref{Eq_FlatConicalStr} is a conical end of total angle $-\pi (\beta + 2)$, where the origin is at infinity.
	\end{itemize}
	In other words, in the case of conical singularity the corresponding special K\"ahler metric is either \emph{locally conical}, \emph{asymptotically cylindrical}, or \emph{asymptotically conical} respectively.
\end{remark}

With this at hand, we can state our main result as follows.

\begin{thm}
	\label{Thm_spKviaHyperbXi_short}
	Let $\Xi$ be a meromorphic cubic differential on a Riemann surface $\Sigma$ (not necessarily compact)  with the divisor $(\Xi)=\sum_{p\in\Sigma}\ord_p\Xi\, \cdot p$.
	Let also $\tilde g$ be a hyperbolic metric on $\Sigma$ representing a divisor $D$.
	Then there is a unique special K\"ahler structure $(g, \nabla)$ on $\Sigma$ whose associated hyperbolic metric and associated cubic form are $\tilde g$ and $\Xi$ respectively.
	Moreover, $(g, \nabla)$ is smooth on $\Sigma_0:=\Sigma\setminus \bigl (\supp (\Xi)\cup \supp D\bigr )$ and the following also holds:
	\begin{enumerate}[(i),topsep=3pt,itemsep=1pt]
		\item A cusp singularity $p$ of $\tilde g$ is a logarithmic singularity of $(g,\nabla)$ of order $\tfrac 12 (\ord_p\Xi +1)$;
		\item A conical singularity $p$ of $\tilde g$ of order $\alpha$ is a conical singularity of $(g,\nabla)$ of order $\tfrac 12 (\ord_p\Xi -\alpha)$;
		\item	$(g,\nabla)$ has a conical singularity of order $\tfrac 12 \ord_p\Xi$ at a point $p\in\supp (\Xi)\setminus \supp D$.
	\end{enumerate}
\end{thm}

A somewhat more precise version of this result is  \autoref{Thm_spKviaHyperbXi_general}, which is proved in \autoref{Sect_spKPeriodMaps}.

We would like to point out that the correspondence of \autoref{Thm_spKviaHyperbXi_short} is pretty much explicit.
To demonstrate this, pick any local holomorphic coordinate $z$  and write $\tilde g = e^{2v}|dz|^2$ and $\Xi = \Xi_0(z)\, dz^3$.
Then the special K\"ahler metric of \autoref{Thm_spKviaHyperbXi_short} is given by
\[
g = 4\,|\Xi_0|e^{-v}|dz|^2.
\]
Using ~\cite{CalliesHaydys17_LocalModels_IMRN_online}*{(9),(11)}, one can also obtain an explicit formula for a connection one-form of $\nabla$ in terms of $v$ and $\Xi_0$.
The details are provided at the end of Section~\ref{Sect_Preliminaries}.


Furthermore, pick integers $k\ge  1, \ \ell\in [0,k]$, a $k$--tuple $\fp =(p_1,\ldots, p_k)$ of pairwise distinct points on $\Sigma$ as well as a   $k$--tuple $\fb = (\b_1,\ldots, \b_k)$ of real numbers.
If $(g,\nabla)$ is a special K\"ahler structure on $\Sigma$ away from $\{ p_1,\dots, p_k \}$, then each $p_j$ is an isolated singularity of the associated cubic form $\Xi$.
It turns out that in general $\Xi$ may have essential singularities at some of $p_j$'s (see Example~\ref{Ex_EssSing}), however in the definition below, we assume that $\Xi$ is meromorphic, i.e., each $p_j$ is a pole of $\Xi$ at worst.
\begin{defn}
\label{defn:moduli}
	We call
	\[
	\begin{aligned}
		\cM_k^\ell (\fp, \fb):=\bigl\{  (g,\nabla) \mid & (g,\nabla)\text{ is a special K\"ahler structure on }\Sigma \ \text{such that }\\
		&\Xi \text{ is meromorphic, } \Xi\not\equiv 0, \text{ and }\ \ord_{p_j} (g,\nabla) =\tfrac 12 \b_j \;\bigr \}/\R_{>0}
	\end{aligned}
	\]
	\emph{the moduli space of special K\"ahler structures with fixed singularities} (or, simply the moduli space of special K\"ahler structures for short),
	where $\ord_{p_j} (g,\nabla)$ is the order of $(g,\nabla)$ at $p_j$,  the first $\ell$ points of $\fp$ are of conic type, and the remaining points are all of logarithmic type.
	In particular,  $\beta_{\ell+1},\cdots, \beta_k$ are integers, if $\ell =0$ all points are  of logarithmic type, whereas for $\ell =k$ all points are of conic type.
	Notice that the group $\R_{>0}$ acts on the set of special K\"ahler structures via $\l\cdot (g,\,\nabla)= (\lambda g,\,\nabla)$.

	In addition, we call
	\[
	\cR_k^\ell (\fp, \fb):=\bigl\{  g\mid \exists\; \nabla  \text{ such that } [g,\nabla]\in \cM_k^\ell(\fp, \fb) \;\bigr \}/\R_{>0}
	\]
	\emph{the moduli space of special K\"ahler metrics.}
\end{defn}

Notice that at this point both $\cM_k^\ell (\fp, \fb)$ and $\cR_k^\ell (\fp, \fb)$ are defined as sets only.
We justify the name by introducing a topology on these sets in Section~\ref{Sect_NecessSuffCond} below.

\begin{thm}
	\label{Cor_NecessCondition_spK}	
	Let $\Sigma$ be a compact Riemann surface of genus $\gamma$.  If $\cM_k^\ell (\fp, \fb)$ is non-empty, then the following inequalities hold:
	\begin{equation}
	\label{Eq_NecessCondMixedSing}
	\begin{gathered}
	4 (\gamma -1 )<\b_1+\cdots+\b_k,\\ [\b_1]+\cdots+[\b_\ell]+\b_{\ell+1}+\cdots+\b_k\leq 6(\gamma -1 )+k-\ell,
	\end{gathered}
	\end{equation}
	where $[\b]$ is the greatest integer not exceeding $\b$.
	
	If $\cM_k^\ell (\fp, \fb)\neq \varnothing$, then it  is homeomorphic to an open dense subset of a sphere of an odd dimension $2N+1$.
	In this case the space $\cR_k^\ell (\fp, \fb)$ is homeomorphic to a Zariski open subset of $\CP^N$.
	In the special case $\ell =k$, i.e., all singularities are of conic type, $\cM_k^k(\fp, \fb)$ is homeomorphic to $S^{2N+1}$ and $\cR_k^k (\fp, \fb)$ is homeomorphic to  $\CP^N$.
	In particular these moduli spaces are compact.
	
	For $\Sigma = \P^1$ the space $\cM_k^\ell (\fp, \fb)$ is non-empty if and only if~\eqref{Eq_NecessCondMixedSing} holds.
	In this case,
	\begin{equation}
		\label{Eq_DimH0L-1}
		N=-6+k-\ell -\sum_{j=1}^k [\b_j].
	\end{equation}
	
\end{thm}

The proof of this theorem can be found in Section~\ref{Sect_NecessSuffCond}.

We also establish necessary and sufficient conditions for the existence of special K\"ahler structures on elliptic curves as well as describe the corresponding moduli spaces in \autoref{Cor_NScondition_EllCurves}.

While proving our main statements we obtain also other results, which may be of some interest.
In particular, as already mentioned above we construct an example of a special K\"ahler structure whose associated cubic form has essential singularities, see  Example~\ref{Ex_EssSing}.
To the best of our knowledge this is the first example of an associated cubic form with an essential singularity.

We also describe all special K\"ahler structures compatible with a fixed metric, see  Section~\ref{Sect_spKStrVsMetrics}.

Furthermore, let  $(g, \nabla)$  be a special K\"ahler structure on a compact Riemann surface with finitely many prescribed singularities.
Then the map  which assigns to  $(g, \nabla)$ the associated cubic form $\Xi$ is injective, see \autoref{NeceSuffCondition_spK} for a more precise statement.
This is surprising, since there is no reason to believe that $\Xi$, which a priori encodes the difference between the Levi--Civita and the flat symplectic connections only, should determine the whole special K\"ahler structure (with prescribed singularities).
Moreover, this is a truly global statement in the sense that the corresponding local statement is clearly false.

Finally, in the last section we construct  compactifications of the moduli spaces $\cM_k^\ell (\fp, \fb)$ and $\cR_k^\ell (\fp, \fb)$ in the case $\ell <k$.

\medskip

\textsc{Acknowledgements.} The authors wish to thank R.~Mazzeo for helpful discussions and an anonymous referee for useful comments.
A.H. was partially supported by the Simons Collaboration on Special Holonomy in Geometry, Analysis, and Physics. B. X. wishes to thank Institute of Mathematics of Freiburg University for her hospitality where he did part of this work with A. H. in February of 2018.
B. X. was partially supported by the National Natural Science Foundation of China (grant nos. 11571330 and 11971450) and
the Fundamental Research Funds for the Central Universities.
Both authors express their gratitude to  the Research Training Group 1821 ``Cohomological Methods in Geometry'' for its support during the work on this project.

\section{Preliminaries}
\label{Sect_Preliminaries}

Let $\Om\subset \C$ be any domain, which is viewed as being equipped
with a holomorphic coordinate $z=x+y i$ and the flat Euclidean metric
$|d z|^2=d x^2+ d y^2$.
We assume that any element of $H^1(\Om;\R)$ can be represented by a co-closed 1-form.

Write a special K\"ahler metric $g$ on $\Om$ in the form
\[
g\, =\, e^{-u}|d z|^2.
\]
Using the global trivialisation of $T\Om$ provided by the real
coordinates $(x,y)$ the connection $\nabla$ is described by its
connection $1$-form $\om^{}_\nabla\in\Om^1\bigl(\Om;\gl(2,\R)\bigr)$.
A computation shows~\cite{CalliesHaydys17_AffSpK2D} that
$\om^{}_\nabla$ can be written in the form
\begin{equation}
\label{Eq_Conn1Form}
\om^{}_\nabla\, = \,
\begin{pmatrix}
\om^{}_{11} & -*\om^{}_{11}\\
*\om^{}_{22} & \phantom{-* }\om^{}_{22}
\end{pmatrix},
\end{equation}
where
\begin{equation}
	\label{Eq_ConnEntriesDmain}
2\,\om^{}_{11}=  e^u(d h +2\,\psi)-d u,\qquad
2\,\om^{}_{22}=  -e^u(d h +2\,\psi)-d u.
\end{equation}
Here $*$ denotes the Hodge star operator with respect to the flat
metric, $h$ is a smooth function, and $\psi$ is a  1-form.
These data are subject to the equation
\begin{equation}
\label{Eq_SpKaehlerOnDomain}
\Delta h = 0,\qquad (d + d^*)\psi = 0,\qquad \Delta u =| d h + 2\psi |^2 e^{2u},
\end{equation}
where $\Delta = \partial_{xx}^2 + \partial_{yy}^2$.
Moreover, given any triple $(h,u,\psi)$ satisfying~\eqref{Eq_SpKaehlerOnDomain} the metric $g=e^{-u}|dz|^2$ together with $\omega_\nabla$, which is given by~\eqref{Eq_Conn1Form} and~\eqref{Eq_ConnEntriesDmain}, constitutes a special K\"ahler structure on $\Om$ (with its complex structure inherited from $\C$).

If $\Om$ is the punctured disc $B_1^*$, any closed and co-closed 1-form can be written as $a\varphi$, where $\varphi$ is a generator of $H^1(B_1^*;\R)$.
For example, we can fix
\[
\varphi= \frac{y\, d x - x\, d y}{x^2 + y^2}=-d\,\big({\rm arg}\, (x+iy)\big).
\]

Hence, a special K\"ahler structure on the punctured disc can be described in terms of solutions of the following equations
\begin{equation}
\label{Eq_SpKaehlerDisc}
\Delta h= 0,\qquad \Delta u = |d h + a\varphi|^2 e^{2u},
\end{equation}
where $h,u\in C^\infty (B_1^*)$ and $a\in\R$.

If $(h, u,a)$ is a solution of~\eqref{Eq_SpKaehlerDisc}, the associated holomorphic cubic form of the corresponding special K\"ahler structure is
\begin{equation*}
	\Xi\, =\, \Xi^{}_0\, d z^3\, =\, \frac 12\Bigl ( \frac a{2z} -\frac {\partial h}{\partial z} i\Bigr)\; d z^3.
\end{equation*}

\begin{rem}
	Tracing through the description of special K\"ahler structures in terms of solutions of~\eqref{Eq_SpKaehlerDisc} as given in~\cite{CalliesHaydys17_AffSpK2D}, it is easy to see that the function $h$ is defined only up to a constant.
	In other words, if $c$ is any real constant,  $(h, u, a)$ and $(h+c, u, a)$ determine equal special K\"ahler structures.
\end{rem}

A straightforward computation shows that $|dh +a\varphi|^2 = 16|\Xi_0|^2 = 16|\Xi|^2$.
Hence, the second equation of~\eqref{Eq_SpKaehlerDisc} can be written as
\begin{equation}
\label{Eq_SpKaehlerDiscXi}
\Delta u = 16\,|\Xi|^2 e^{2u},
\end{equation}
which implies in particular that the Gaussian curvature of the special K\" ahler metric $g=e^{-u}|dz|^2$ equals $8|\Xi|_g^2\geq 0$.
Furthermore, write $\Xi_0 (z) = \mathring \Xi_0(z) + Az^{-1}$, where $A\in\C$ is the residue of $\Xi_0$ at the origin and denote by $\Eta$ a primitive of $\mathring\Xi_0$.
Notice that $\Eta$ is well-defined up to a constant.
Define
\begin{equation}
	\label{Eq_haViaXi}
h :=-4\Im \Eta -4\Im A\, \log |z|\qquad\text{and}\qquad a := 4\Re A.
\end{equation}
Using $\partial_z\Im \Eta = \tfrac 1{2i}\partial_z\Eta$ we compute
\[
\frac 12\Bigl ( \frac a{2z} -\frac {\partial h}{\partial z} i\Bigr) = \Xi_0(z).
\]
The upshot of this computation is that $\Xi_0$ determines and is determined by $h$ and $a$.

\medskip

Slightly more generally, let $\Om=\tilde\Om\setminus\{ p_1,\dots, p_k \}$, where $\tilde\Om$ is a simply connected domain in $\C$.
Then any closed and co-closed 1-form $\psi$ representing a non-trivial cohomology class can be written as $\sum_{j=1}^k a_k\varphi_{k}$, where
$\varphi_k := -d\,\big({\rm arg}\, (z-p_k)\big)$.
It is easy to see that the above discussion can be repeated verbatim in this case too leading to the following result.

\begin{proposition}
	\label{Prop_spKinTermsOfKWXi}
	Let $\Om = \tilde{\Om}\setminus\{ p_1,\dots,p_k \}$, where $\tilde\Om$ is a simply connected domain in $\C$.
	Any pair $(u,\Xi)$ satisfying~\eqref{Eq_SpKaehlerDiscXi} determines a  special K\"ahler structure on $\Om$ such that the corresponding associated cubic form is $\Xi$. Conversely, any special K\"ahler structure on $\Om$ determines a solution of~\eqref{Eq_SpKaehlerDiscXi}.\qed
\end{proposition}

%
%

For the sake of clarity, let us spell the correspondence in the above proposition.
Thus, if $(u,\; \Xi=\Xi_0\, dz^3)$ is a solution of~\eqref{Eq_SpKaehlerDiscXi}, put $g=e^{-u}|dz|^2$.
Also, write $\Xi_0 (z) = \mathring\Xi_0(z) + \sum_j A_j(z-p_j)^{-1}$, where $A_j$ is the residue of $\Xi_0$ at $p_j$.
If $H$ is a primitive of $\mathring\Xi_0$, put
\begin{equation*}
h:= -4\Im H - 4\sum_{j=1}^k  \bigl (\Im A_j\bigr ) \log |z-p_j|\qquad\text{and}\qquad a_j:= 4\Re A_j.
\end{equation*}
Then the corresponding special K\"ahler structure is given by~\eqref{Eq_Conn1Form} and~\eqref{Eq_ConnEntriesDmain} with $\psi = \sum_j a_j\varphi_j$.

\section{Special K\"ahler structures and the period maps}
\label{Sect_spKPeriodMaps}

Let $(\Sigma, g, I, \om, \nabla)$ be a special K\"ahler structure, where $\dim_\C \Sigma =n$.
Denote by $\cU$ the corresponding affine structure.
This means that $\cU$ is a covering of $M$ by open sets; Moreover, each $U\in\cU$ is equipped with  a $2n$-tuple of holomorphic functions $(z_1,\dots, z_n; w_1,\dots, w_n)$, where $(z_1,\dots, z_n)$ and $(w_1,\dots, w_n)$ are conjugate special holomorphic coordinates on $U$~\cite{Freed99_SpecialKaehler}.
If $\tilde U\in\cU$ is another open set equipped with $(\tilde z, \tilde w)$, then we have a relation
\[
\begin{pmatrix}
z\\
w
\end{pmatrix}
=P
\begin{pmatrix}
\tilde z\\
\tilde w
\end{pmatrix}
+
\begin{pmatrix}
a\\
b
\end{pmatrix},
\]
where $P\in \Sp(2n;\R)$ and $a,b\in\C^n$ are some constants.

Denote
\[
\tau_{jk}=\frac{\partial w_k}{\partial z_j}.
\]
Then the matrix $\tau:=(\tau_{jk})$ is symmetric and $\Im\tau$ is positive definite.
In fact, $\om = \frac i2\sum \Im \tau_{jk} dz_j\wedge d\bar z_k$.
In particular, we have a holomorphic map
\[
\tau\colon U\to \Sieg_n:=\bigl\{  Z\in M_n(\C)\mid Z^t=Z,\ \Im Z\text{ is positive definite}\,\bigr\}
\]
whose target space is the Siegel upper half-space.

Recall that the group $Sp(2n,\R)$ acts on $\Sieg_n$ via
\[
P\cdot Z = (AZ + B)(CZ + D)^{-1},\qquad\text{where }
P=\begin{pmatrix}
A & B\\
C& D
\end{pmatrix},
\]
and the unique $Sp(2n,\R)$-invariant metric is given by $g_{\Sieg_n}=\tr \bigl ( (Y^{-1} dZ) (Y^{-1}d\bar Z) \bigr)$, where $Y=\Im Z$.

If $\tilde \tau$ is a map corresponding to the chart $\tilde U$, then the corresponding period maps are related by
\[
\tau = (D\tilde \tau + C)(A\tilde \tau +B)^{-1}=\tilde P\cdot\tau,\qquad\text{where}\quad
\tilde P=
\begin{pmatrix}
D & C\\
B & A
\end{pmatrix}\in \Sp(2n,\R).
\]
Hence,  $\tau^*g_{\Sieg_n}$ does not depend on the choice of an affine patch.
As already explained in the introduction, if $n=1$,  $\tilde g:=\tau^*g_{\Sieg_1}$ is a constant negative curvature metric, which we call the associated hyperbolic metric.

While the pull-back metric is defined in any dimension, the case  $n=1$ has some special features.
Indeed, in this case $\Sigma$ is a Riemann surface,  $\Sieg=\Sieg_1$ is the upper half-plane so that $\tau$ is a local biholomorphism except perhaps at isolated points.
Hence, $\tilde g$ is non-degenerate on $\Sigma$  outside of some discrete subset.
Moreover, the subset where $\tilde g$ degenerates is easy to describe, see \autoref{Prop_DegenLocusAssHypMetric} below.

More importantly, in the case $n=1$  the metric $g_{\Sieg_1}$ coincides with the standard hyperbolic metric $\bigl (\Im z\bigr)^{-2}|dz|^2$.
Hence, the pull-back metric $\tilde g$ is also hyperbolic where it is non-degenerate.

\begin{rem}
	Recall, that a holomorphic map $\tau\colon \Sigma\to\Sieg$, which may be multi-valued, is called a developing map of a hyperbolic metric $\tilde g$, if $\tilde g =\tau^*g_{\Sieg}$.
	Hence, the very definition yields that the period map of a special K\"ahler structure  is a developing map of the associated hyperbolic metric.
\end{rem}

\begin{example}
	Consider the following local example: $\Sigma$ is the punctured unit disc in $\C$ equipped with the metric $g=-\log |z|\, |dz|^2$, which is special K\"ahler.
	Then $z$ is a special holomorphic coordinate with the conjugate given by $w =2i(z\log z - z)$.
	Hence, the period map is $\tau =2 i\log z$.
	Of course, $\tau$ is multivalued, but all values of $\tau$ are related by M\"obious transformations and therefore $\tau^*g_{\Sieg}$ is well defined and equals $(|z|\log |z|)^{-2} |dz|^2$, which is the standard Poincar\'{e} metric on the punctured disc.
\end{example}

\begin{example}
	Let $\Sigma$ be the upper half-plane $\cH$ equipped with the following special K\"ahler structure~\cite{Freed99_SpecialKaehler}*{Rem.\ 1.20}:
	\[
	g=y\, |dz|^2,\qquad \om_\nabla =\frac 1y
	\begin{pmatrix}
	dy & dx\\
	0 & 0
	\end{pmatrix},
	\]
	where $z=x+yi$ is a coordinate on $\cH$.
	
	It is easy to check that $(-iz, -\tfrac i2 z^2)$ is a pair of special holomorphic conjugate coordinates.
	Hence, $\tau(z)=z$, which means that $\tau^*g_{\Sieg}=g_{\Sieg}$.
\end{example}

\medskip

It will be useful below to have a relation between $\Xi$ and $\tau$.
Thus, if $Z$ is a special holomorphic coordinate, we have
\begin{equation}
	\label{Eq_XiInSpecCoord}
\Xi = \frac 14\,\frac {d\tau}{dZ}\, dZ^3.
\end{equation}
Then, for an arbitrary holomorphic coordinate $z$ we obtain
\[
\Xi = \Xi_0\, dz^3 = \frac 14\,\frac {d\tau}{dZ}\, dZ^3 = \frac 14\,\frac {d\tau}{dz}\frac {dz}{dZ}\left (\frac {dZ}{dz}\right )^3 dz^3 = \frac 14\,\frac {d\tau}{dz}\left (\frac {dZ}{dz}\right )^2dz^3,
\]
which yields in turn
\begin{equation}
\label{Eq_RelationTauXiZ}
\frac {d\tau}{dz} = 4\,\Xi_0\cdot \left (\frac {dZ}{dz}\right )^{-2}.
\end{equation}

Notice in particular, that we have the following statement, which will be useful below.
\begin{proposition}
	\label{Prop_DegenLocusAssHypMetric}
	Let $p$ be a regular point of a special K\"ahler structure on a Riemann surface.
	Then the associated hyperbolic metric degenerates at $p$ if and only if $\Xi (p) =0$.\qed
\end{proposition}

The next result is the key ingredient in the proof of our main result, \autoref{Thm_spKviaHyperbXi_short}.

\begin{lem}\label{Lem_spKandHyp_loc}
	Let $\Om$ be as in~\autoref{Prop_spKinTermsOfKWXi}.
	\begin{enumerate}[(i)]
		\item \label{It_AssocHyperbMetr}
		Let $\bigl (g =e^{-u}|dz|^2,\;\nabla\bigr )$ be a special K\"ahler structure on $\Om$.
		Then the associated hyperbolic metric, which is defined on $\Omega\setminus \Xi^{-1}(0)$, is given by $\tilde g = e^{2v}|dz|^2$, where
		\begin{equation}
		\label{Eq_UVrelation}
		v = u +\log |\Xi_0| +2\log 2.
		\end{equation}
		\item \label{It_HyperbMandXi_spK}
		Given any hyperbolic metric $\tilde g =e^{2v}|dz|^2$ and any holomorphic cubic form $\Xi=\Xi_0(z)\, dz^3$ on $\Om$, there is a unique special K\"ahler structure $(g,\nabla)$ on $\Om\setminus \Xi^{-1}(0)$ such that $g=e^{-u}|dz|^2$, where $u$ is determined by~\eqref{Eq_UVrelation}, and $\Xi$ is the associated cubic form.
	\end{enumerate}
\end{lem}

\begin{rem}
	We would like to point out that in the statement of~\autoref{Lem_spKandHyp_loc}, the domain $\Om$ is allowed to have no punctures, i.e., $k=0$ is allowed.
\end{rem}

\begin{proof}[Proof of~\autoref{Lem_spKandHyp_loc}.]
	Notice that since $\Xi_0$ is holomorphic, $\log |\Xi_0|$ is harmonic on $\Om\setminus \Xi^{-1}(0)$.
	Since by \autoref{Prop_spKinTermsOfKWXi} the pair $(u,\Xi_0)$ satisfies~\eqref{Eq_SpKaehlerDiscXi},  for  $v: = u +\log |\Xi_0|+2\log 2$ we have
	\[
	\Delta v = \Delta u = 16\, |\Xi_0|^2 e^{2v - 2\log |\Xi_0| -4\log 2}=e^{2v}.
	\]
	Hence, $\tilde g = e^{2v}|dz|^2$ is a metric of constant curvature $-1$ on $\Om\setminus \Xi^{-1}(0)$.
	
	Furthermore, we claim that $\tau^*g_\Sieg = \tilde g$.
	To see this,  notice that if $Z$ is a special holomorphic coordinate (in a neighbourhood of some point), we have
	\[
	g = e^{-u}|dz|^2 = \bigl ( \Im\tau \bigr) |dZ|^2 = \bigl ( \Im\tau \bigr) \left |\partial_zZ \right|^2 |dz|^2 = \bigl ( \Im\tau \bigr)\frac{4\, |\Xi_0|}{|\partial_z\tau|}\, |dz|^2.
	\]
	Here the last equality follows from~\eqref{Eq_RelationTauXiZ}.
	Hence,
	\[
	u=\log  |\partial_z\tau|-\log\bigl ( \Im\tau\bigr)  -\log |\Xi_0|-2\log 2\qquad\Leftrightarrow\qquad v= \log  |\partial_z\tau|-\log\bigl ( \Im\tau\bigr),
	\]
	which yields in turn
	\[
	\tilde g = e^{2v}|dz|^2 = \frac{|\partial_z\tau|^2}{\bigl ( \Im\tau\bigr)^2} |dz|^2 = \tau^*g_{\Sieg_1}.
	\]
	This clearly proves~\textit{\ref{It_AssocHyperbMetr}}.
	
	The last part,~\textit{\ref{It_HyperbMandXi_spK}}, is obtained essentially by reading the above computation backwards.
	That is, if $\tilde g = e^{2v}|dz|^2$ is a metric of constant curvature $-1$, we have $\Delta v = e^{2v}$.
	Using this, it is easy to check that for any holomorhic function $\Xi_0$ the function
	\begin{equation*}
		u:=v-\log |\Xi_0| - 2\log 2
	\end{equation*}
	satisfies~\eqref{Eq_SpKaehlerDiscXi}.
	Appealing to \autoref{Prop_spKinTermsOfKWXi}, we obtain~\textit{\ref{It_HyperbMandXi_spK}}.
\end{proof}

\begin{corollary}
	\label{Thm_SpKHyperbLoc}
	Let $g$ be a special K\"ahler metric on the punctured disc $B_1^*$ such that the associated holomorphic cubic form $\Xi$ has order $n\in\Z$ at the origin.
	Let $\tilde g$ be the associated hyperbolic metric.
	Then the following holds:
	\begin{enumerate}[(i),topsep=3pt,itemsep=-3pt]
		\item
		  \label{It_ConicalCase}
		 $g$ is conical of order $\b/2$ if and only if $\tilde g$ is conical of order $n-\b\in (-1, +\infty)$, i.e.,
		\begin{equation*}
			\label{Eq_ConeImpliesCone}
		g = r^\b \bigl (C + o(1)\bigr )\,|dz|^2\qquad\Longleftrightarrow\qquad \tilde g = r^{2(n-\b)} \bigl (C' + o(1)\bigr )\,|dz|^2;
		\end{equation*}
		\item
			\label{It_LogCase}
		$g$ has a logarithmic singularity if and only if $\tilde g$ has a cusp, i.e.,
		\begin{equation*}
			\pushQED{\qed}
		g= -r^{n+1}\log r  \bigl (C + o(1)\bigr )\, |dz|^2 \qquad\Longleftrightarrow\qquad  \tilde g = \frac {C' + o(1)}{(r\log r)^2} \, |dz|^2.\qedhere
		\popQED
		\end{equation*}
	\end{enumerate}
\end{corollary}

The inequality $n-\beta>-1$ claimed in~\textit{\ref{It_ConicalCase}}  has been established in \cite[Thm.\,1.1]{Haydys15_IsolSing_CMP}.
Of course, this also follows from the classifaction of isolated singularities for metrics of constant negative curvature.


\begin{thm}
	\label{Thm_spKviaHyperbXi_general}
	Let $\Sigma$ be a Riemann surface (not necessarily compact) and $\Sigma_0\subset \Sigma$ be an open subset.
	For any holomorphic cubic form $\Xi$  and any smooth hyperbolic metric $\tilde g$  on $\Sigma_0$ there is a unique special K\"ahler structure  $(g,\nabla)$ on $\Sigma_0\setminus\Xi^{-1}(0)$ whose associated hyperbolic metric and associated cubic form are $\tilde g$ and $\Xi$ respectively.
	
	If $\Xi$ is meromorphic on $\Sigma$ with the divisor $(\Xi)=\sum_{p\in\Sigma}\ord_p\Xi\, \cdot p$ and $\tilde g$ represents a divisor $D$, then for the special K\"ahler structure $(g,\nabla)$ on $\Sigma_0:=\Sigma\setminus\bigr (\supp (\Xi)\cup \supp D\bigl )$ as above the following holds:
	\begin{enumerate}[(i),topsep=3pt,itemsep=1pt]
		\item A cusp singularity $p$ of $\tilde g$ is a logarithmic singularity of $(g,\nabla)$ of order $\tfrac 12 (\ord_p\Xi +1)$;
		\item A conical singularity $p$ of $\tilde g$ of order $\alpha$ is a conical singularity of $(g,\nabla)$ of order $\tfrac 12 (\ord_p\Xi -\alpha)$;
		\item	$(g,\nabla)$ has a conical singularity of order $\tfrac 12 \ord_p\Xi$ at a point $p\in\supp (\Xi)\setminus \supp D$.
	\end{enumerate}
\end{thm}
\begin{proof}
	Pick a point $p\in\Sigma$ and an open set $U$ together with a holomorphic coordinate $z$ centered at $p$.
	If $p\notin \supp (\Xi)\cup \supp D$,  we may think of $U$ as a disc $\{ |z|<1 \}$.
	Otherwise, $U$ can be chosen to be the punctured disc.
	
	By~\autoref{Lem_spKandHyp_loc}, $\tilde g=e^{2v}|dz|^2$ and $\Xi=\Xi_0(z)\, dz^3$ determine a unique special K\"ahler structure $(g,\nabla)$ on $U$, where $g=e^{-u}|dz|^2$ with $u=v-\log |\Xi_0| -2\log 2$.
	Moreover, $u$ satisfies~\eqref{Eq_SpKaehlerDiscXi}.
	Since this construction of $(g,\nabla)$ involves a local coordinate, $(g,\nabla)$ a priori depends on this choice.
	We prove, however, that it is in fact immaterial, i.e., different choices yield equal special K\"ahler structures.
	
	To this end, choose another holomorphic coordinate $\hat z$ on $U$.
	If $\hat z = f(z)$, where $f$ is holomorphic, the local representations $\hat \Xi_0(\hat z)\, d\hat z^3$ and $\Xi_0(z)\, dz^3$ of $\Xi$ are related by $\hat \Xi_0(\hat z) = \Xi_0(z) \bigl ( f'(z) \bigr)^{-3}$.
	Also, for the flat metric $ g_1 =|d\hat z|^2$ and the corresponding Laplacian $ \Delta_1 = \partial^2_{\hat x\hat x} + \partial^2_{\hat y\hat y}$ we have
	\[
	 g_1 = |f'(z)|^2 |dz|^2\qquad\text{and}\qquad \Delta_1 = |f'(z)|^{-2}\Delta.
	\]
	
	Multiply~\eqref{Eq_SpKaehlerDiscXi} by $|f'(z)|^{-2}$ to obtain
	\begin{equation*}
		\Delta_1 u = |f'(z)|^{-2}\, |\Xi_0(z)|_0^2\; e^{2u} =
		|f'(z)|^4\, |\hat\Xi_0(\hat z)|^2_0\; e^{2u},
	\end{equation*}
	where the subscript ``$0$'' indicates the norm induced by $|dz|^2$.
	Furthermore, for $\hat u:=u + 2\log |f'(z)|$ we compute
	\[
	\Delta_1\, \hat u = \Delta_1\, u  = |\hat \Xi_0|^2_{1}\; e^{2\hat u}.
	\]
	Hence, for the unique special K\"ahler structure $(\hat g, \hat \nabla)$ determined by $(\hat u, \Xi)$ in the coordinate $\hat z$ as in~\autoref{Prop_spKinTermsOfKWXi}, we have
	\[
	\hat g = e^{-\hat u}|d\hat z|^2 = e^{-u} |dz|^2 = g.
	\]
	 Since a special K\"ahler metric and the associated cubic form determine the flat symplectic connection uniquely, we conclude that $(g, \nabla)$ and $(\hat g, \hat \nabla )$ coincide (more precisely, this means $(g,\nabla) = f^*(\hat g,\hat \nabla)$).
	 By the construction,  $(\hat g,\hat \nabla)$ is the special K\"ahler structure determined by $\Xi$ and the hyperbolic metric
	 \[
	 \exp\bigl (2\hat u +2\log |\hat \Xi_0(\hat z)|)\;\bigr|d\hat z|^2 = e^{2v} |dz|^2 = \tilde g,
	 \]
	 where the above equality follows from the definition of $\hat u$.
	  Thus, the choice of the local coordinate used in~\autoref{Prop_spKinTermsOfKWXi} is immaterial as claimed.
	 This proves the existence of a special K\"ahler structure for given $\Xi$ and $\tilde g$.
	
	 The uniqueness of the special K\"ahler structure corresponding to $(\tilde g,\Xi)$ follows immediately from the corresponding local statement.
	 The other properties claimed follow directly from~\autoref{Thm_SpKHyperbLoc}.
\end{proof}

\begin{example}[A special K\"ahler structure whose associated cubic form has an essential singularity]
	\label{Ex_EssSing}
	Let $\Xi (z):=e^{1/z}dz^3$ be a cubic holomorphic form on $\C^*$.
	$\Xi$ may be thought of as a holomorphic cubic form on $\P^1$ with two singularities: one essential and the other one of degree $-6$.
	Pick a hyperbolic metric singular at any $3$ points $w_1,w_2,w_3\in \P^1$.
	By \autoref{Thm_spKviaHyperbXi_general} we obtain a special K\"ahler structure on $\P^1$ with at least three and at most five singularities depending on the number of points in $\{ w_1, w_2, w_3 \}\cap \{ 0,\infty \}$ such that $\Xi$ is the associated cubic form.
	To the best of our knowledge, this is the first example of a special K\"ahler structure whose associated cubic form has essential singularities.
\end{example}

Let $(g,\nabla)$ be a special K\"ahler structure on the punctured disc whose associated cubic form $\Xi=\Xi_0\, dz^3$ has an essential singularity at the origin.
Assume for simplicity of exposition that the origin is a regular point for the associated hyperbolic metric $\tilde g = e^{2v}|dz|^2$.
The existence of such structures follows by  \autoref{Thm_spKviaHyperbXi_general} just like in the example above.
By~\eqref{Eq_UVrelation} we have
\begin{equation*}
	g = e^{-u}|dz|^2 = 4e^{-2v} |\Xi_0|\, |dz|^2.
\end{equation*}
Notice that $e^{-2v}$ has a positive limit at the origin, whereas by the great Picard Theorem $|\Xi_0|$ takes any positive value near the origin.
Hence, in this case the behavior of $g$ is highly irregular near the origin.

\section{Special K\"ahler metrics versus special K\"ahler structures}
\label{Sect_spKStrVsMetrics}

\autoref{Thm_spKviaHyperbXi_general} allows us to construct  inequivalent special K\"ahler structures such that the corresponding Riemannian metrics are equal.
Indeed, fix a pair $(\tilde g,\Xi)$ as in \autoref{Thm_spKviaHyperbXi_general} and let $g$ be the corresponding special K\"ahler metric.
It is then clear from~\eqref{Eq_UVrelation}  that the pair $(\tilde g, \l\, \Xi)$ leads to the metric $|\l|\cdot g$,  where  $\l\in \C^*$.
Hence, specializing to $|\l|=1$ we obtain a family of special K\"ahler structures parameterized  by $S^1$ such that all corresponding Riemannian metrics are equal.

\begin{example}
	Fix arbitrarily a hyperbolic metric $\tilde g$ on the punctured unit disc $B_1^*$.
	Choose a holomorphic cubic differential $\Xi$ on $B_1^*$ such that $\Xi$ is of order $-3$ at the origin.
	Observe that the leading coefficient $\xi_{-3}$ in the expansion $\Xi_0 (z)= \xi_{-3}\,{z^{-3}}+ {\xi_{-2}}\,{z^{-2}}+\dots$
	is independent of the choice of a local coordinate.
	Hence, the family $\{\, \l\,\Xi\, \mid |\l| =1\}$ consists of holomorphic cubic differentials that are pairwise inequivalent even up to a change of coordinates.
	Hence, for the corresponding family of special K\"ahler structures $(g, \nabla_\l)$  the metric is independent of $\l$ and the corresponding structures are pairwise inequivalent.
\end{example}

\begin{proposition}
	\label{Prop_spKwithFixedMetric}
	Let $\Sigma$ be a Riemann surface.
	\begin{enumerate}[(i), itemsep=1pt, topsep=-2pt]
		\item\label{It_AssocCubicFormsWithAFixedMetric}
		Let $(g,\nabla)$ and $(\hat g, \hat\nabla)$ be two special K\"ahler structures  on $\Sigma$, whose associated cubic forms are  $\Xi$ and $\hat \Xi$ respectively. If $g=\hat g$, then
		\begin{equation}
		\label{Eq_XiHatXi}
		\hat \Xi = \l\cdot \Xi,
		\end{equation}
		where $\l\in\C$ is of absolute value $1$;
		\item \label{It_spKFamilyWithfixedMetric}
		If $(g,\nabla)$ is a special K\"ahler structure on $\Sigma$ whose associated cubic form is $\Xi$, then for each $\l\in S^1$ there is a unique special K\"ahler structure $(g, \nabla_\l)$, whose associated cubic form is $\l\,\Xi$.
	\end{enumerate}	
\end{proposition}
\begin{proof}
	Clearly, to prove \textit{\ref{It_AssocCubicFormsWithAFixedMetric}} it is enough to check~\eqref{Eq_XiHatXi} in a neighborhood of a regular point $p\in\Sigma$.
	Thus, let $z$ be a local holomorphic coordinate in a neighborhood $U$ of $p$.
	
	If $(g,\nabla)$ and $( \hat g, \hat\nabla)$ are two special K\"ahler structures with the associated cubic forms $\Xi = \Xi_0\, dz^3$ and  $\hat\Xi = \hat \Xi_0\, dz^3$ respectively such that $g=\hat g$,  then~\eqref{Eq_SpKaehlerDiscXi} implies $|\Xi_0| = |\hat \Xi_0|$.
	Since both $\Xi_0$ and $\hat\Xi_0$ are holomorphic on $U$, there is $\l\in S^1\subset  \C, $  such that $\Xi_0 = \l\cdot\hat \Xi_0$.
	This proves~\textit{\ref{It_AssocCubicFormsWithAFixedMetric}}.
	
	Claim \textit{\ref{It_spKFamilyWithfixedMetric}} follows from \autoref{Thm_spKviaHyperbXi_general} by setting $\tilde g$ to be the associated hyperbolic metric of $(g,\nabla)$.
\end{proof}

\section{A necessary and sufficient condition for the existence of special K\"ahler structures on compact Riemann surfaces}
\label{Sect_NecessSuffCond}

Just like in the introduction, pick integers $k\ge 1, \ \ell\in [0,k]$, a $k$--tuple $\fp =(p_1,\ldots, p_k)$ of pairwise distinct points on $\Sigma$ as well as a   $k$--tuple $\fb = (\b_1,\ldots, \b_k)$ of real numbers, where $\beta_{\ell+1},\cdots, \beta_k$ are integers.
Denote
\[
\begin{gathered}
D=D(\fp, \fb):=-\sum_{j=1}^\ell [\b_j] p_j - \sum_{j=\ell+1}^k (\b_j -1)p_j, \qquad L=L(\fp, \fb):=\cO \bigl ( K_\Sigma^3 + D\bigr ), \qand\\
H(\fp,\fb): =\Bigl\{ \Xi\in \rH^0 (L)\mid \Xi\neq 0,\  \ord_{p_j}\Xi = \b_j -1\ \text{for } \ell+1\le j\le k\Bigr \},
\end{gathered}
\]
where $K_\Sigma$ is the canonical bundle of $\Sigma$.
Recall also that $[\beta]$ denotes the greatest integer not exceeding $\beta$.
In other words, $H(\fp,\fb)$ consists of all non-trivial meromorphic cubic differentials $\Xi$ which are holomorphic  on  $\Sigma\setminus \{p_1,\cdots, p_k\}$ and satisfy
\begin{equation}
\label{Eq_OrdersOfXi}
\ord_{p_j}\Xi \ge [\b_j]\ \text{ if }j\le\ell\qquad\text{and}\qquad  \ord_{p_j}\Xi = \b_j -1\ \text{ if }j >\ell.
\end{equation}

For each $j>\ell$ choose a local holomorphic coordinate $z_j$ centered at $p_j$ and consider the holomorphic map
\begin{equation*}
	f_j\colon \rH^0(L)\to \C\quad\text{defined by}\quad
	f_j(\Xi) = a_j\  \Leftrightarrow\  \Xi = \Bigl ( a_jz_j^{\beta_j-1} +\dots \Bigr) \, dz_j^3,
\end{equation*}
where dots denote the higher order terms.
Then $H(\fp,\fb)$ is the subset of $\rH^0(L)$ where each $f_j$ does not vanish.
Hence, $H(\fp,\fb)$ is Zarisky open.

\begin{thm}
	\label{NeceSuffCondition_spK}	
	Let $\Sigma$ be a compact Riemann surface of genus $\gamma$.
	Then $\cM_k^\ell (\fp, \fb)\neq\varnothing$ if and only if the following two conditions hold{\rm :}
	\begin{enumerate}[(i),topsep=3pt,itemsep=1pt]
		\item
		\label{It_NCfromGB}
		$4 (\gamma -1 )<\b_1+\cdots+\b_k$;
		\item
		\label{It_OrdersOfXi}
		 $H(\fp,\fb)\neq\varnothing$.
	\end{enumerate}
Moreover, the map that assigns to a special K\"ahler structure $(g,\nabla)$ as above its associated cubic form $\Xi\in H(\fp,\fb)$ is a bijection.
\end{thm}

\begin{proof}
	If $\cM_k^\ell (\fp, \fb)\neq\varnothing$, then by the quantitative relationship between the special K\" ahler metric and the associated cubic form
in Corollary \ref{Thm_SpKHyperbLoc}, the associated cubic form $\Xi$ of any special K\"ahler structure $(g, \nabla)$ such that $[g,\nabla]\in \cM_k^\ell (\fp, \fb)$ lies in $H(\fp,\fb)$, hence \textit{\ref{It_OrdersOfXi}} holds.

	Furthermore, since $\ord_{p_j}\Xi-\b_j\ge -1$, the associated hyperbolic metric $\tilde g$ has either a conical singularity with positive angle or a cusp at each $p_j$.
	Let $p_{k+1},\dots, p_m$ be further points on $\Sigma$ such that $\tilde g$ is singular.
	By 	\autoref{Thm_SpKHyperbLoc}, \textit{\ref{It_ConicalCase}}  each $p_j$ with $j\ge k+1$ is conical singularity of $\tilde g$ of order $\ord_{p_j}\Xi >0$.
	Hence, by the Gau\ss--Bonnet theorem applied to $\tilde g$ we have
	\[
	\sum_{j=1}^k \bigl (\ord_{p_j}\Xi -\b_j\bigl ) + \sum_{j=k+1}^m\ord_{p_j}\Xi + 2-2\gamma <0.
	\]
	Hence, we obtain
	\begin{equation*}
		\sum_{j=1}^k \b_j > \sum_{j=1}^m \ord_{p_j}\Xi + 2-2\gamma = 4(\gamma -1).
	\end{equation*}
	
	It remains to show that \textit{\ref{It_NCfromGB}} and \textit{\ref{It_OrdersOfXi}} yield a special K\"ahler structure.
	Indeed, notice that \textit{\ref{It_NCfromGB}} and \textit{\ref{It_OrdersOfXi}}  imply
	\[\sum_{j=1}^k\, \Big({\rm ord}_{p_j}\,\Xi-\b_j\Big)+ \sum_{\substack{\Xi(Q)=0\\ Q\notin \{p_1,\cdots,p_k\}}} \ord_Q\,\Xi+2-2\gamma<0.\]
	Hence, by \cite{Heins62_OnClassOfConfMetrics} there exists a hyperbolic metric $\tilde g$ which has conical singularities at each zero $q$ of $\Xi$ of order $\ord_q\Xi$, and has either a conical singularity or a cusp at each $p_j$ for all $1\leq j\leq k$.
	The proof is finished by appealing to \autoref{Thm_spKviaHyperbXi_general}.
\end{proof}

%
%

Denote by $\pi\colon \cM_k^\ell (\fp, \fb) \to \cR_k^\ell (\fp, \fb)$ the natural projection, which has been studied in Section~\ref{Sect_spKStrVsMetrics}.
In particular, each fiber of $\pi$ is isomorphic to the circle.

We have the commutative diagram
\begin{equation}
\label{Diagr_Moduli}
\begin{CD}
  \cM_k^\ell (\fp, \fb)@>\Xi>>H(\fp,\fb)/\R_{>0}\\
@V\pi VV @VV V\\
\cR_k^\ell (\fp, \fb) @>\xi >> H(\fp,\fb)/\C^*,
\end{CD}
\end{equation}
where slightly abusing notations $\Xi$ stays for the map, which assigns to a special K\"ahler structure its associated cubic form, and $\xi$ is just the induced map.
By \autoref{Prop_spKwithFixedMetric} and  \autoref{NeceSuffCondition_spK} both $\Xi$ and $\xi$  are bijections.
This can be used to define topologies on   $\cM_k^\ell (\fp, \fb)$ and $\cR_k^\ell (\fp, \fb)$.
Indeed, $H(\fp,\fb)$ is naturally a subset of a vector space $\rH^0 ( L)$, which can be equipped with a topology by introducing a Hermitian inner product (notice that the origin is not contained in $H(\fp,\fb)$).

%
\begin{proof}[\textbf{Proof of \autoref{Cor_NecessCondition_spK}}]
The proof consists of the following parts.
\begin{enumerate}
\item
The first inequality of ~\eqref{Eq_NecessCondMixedSing}  coincides with Theorem
\ref{NeceSuffCondition_spK} \textit{\ref{It_NCfromGB}}.
The second one of \eqref{Eq_NecessCondMixedSing} follows by combining the following facts:
 $\deg\,K_\Sigma^3=6\gamma-6$, $\Xi$ is holomorphic outside $\{p_j:1\leq j\leq k\}$, and \eqref{Eq_OrdersOfXi}.

 \item By \autoref{NeceSuffCondition_spK}, $\cM_{k}^\ell (\fp, \fb)$ is homeomorphic to an open subset of the unit sphere in the complex vector space $\rH^0(L)$. Hence, $\dim\cM_{k}^\ell (\fp, \fb)$ is odd if $\cM_{k}^\ell (\fp, \fb)\neq \varnothing$.
 Likely, $\cR_{k}^\ell (\fp, \fb)$ is homeomorphic to a Zariski open subset of $\P \rH^0(L)$.
 Moreover, if $\ell =k$, then $H(\fp, \fb) = \rH^0(L)\setminus \{ 0 \}$.
 Hence, if $\rH^0(L)$ is non-trivial,  $\cM_{k}^\ell (\fp, \fb)$ and  $\cR_{k}^\ell (\fp, \fb)$ can be identified with $S(\rH^0(L))$ and $\P \rH^0(L)$ respectively.

\item
We prove that
the space $\cM_k^\ell (\fp, \fb)$ is non-empty for $\Sigma = \P^1$ provided~\eqref{Eq_NecessCondMixedSing} holds.
Indeed, given a ${\Bbb Z}$-divisor $D=m_1[z_1]+\cdots+m_n\,[z_n]$ ($z_1,\cdots, z_n\in {\Bbb C}\subset \P^1$) of degree $-6$, the meromorphic cubic differentials whose associated divisor coincides with $D$ have the form
${\rm Const.}\,(z-z_1)^{m_1}\cdots (z-z_n)^{m_n}\,dz^3$. If $z_1=\infty\in\P^1$, then the cubic differentials equals ${\rm Const.}\,(z-z_2)^{m_2}\cdots (z-z_n)^{m_n}\,dz^3$.
Hence, given a $k$-tuple of points on $\Sigma = \P^1$ as well as a $k$-tuple $\fb=(\b_1, \dots, \b_k)$ of real numbers such that~\eqref{Eq_NecessCondMixedSing} holds, there exists a meromorphic cubic differential $\Xi$ on $\P^1$ such that~\eqref{Eq_OrdersOfXi} holds and $\Xi$ is holomorphic on $\P^1\setminus\{ p_1,\dots, p_k \}$. It follows from Theorem \ref{NeceSuffCondition_spK} that the space $\cM_k^\ell (\fp, \fb)$ is non-empty.
	
\item Finally, by Riemann-Roch,  in the case $\Sigma = \P^1$ the complex dimension of ${\rm H}^0(L)= {\rm H}^0(K_\Sigma^3+D)$ is $N+1$, where $N$ is given by~\eqref{Eq_DimH0L-1}.
\end{enumerate}
\end{proof}

\begin{remark}
	It is clear from the proof of \autoref{Cor_NecessCondition_spK} that the case $\Sigma=\P^1$ is somewhat special due to the fact that it is easy to describe when $\rH^0(L)$ is non-trivial.
	In general, the non-triviality of $\rH^0(L)$ depends on the complex structure on $\Sigma$  and the fixed singularities $(\fp,\fb)$ of the special K\"ahler structures under consideration.
\end{remark}

\begin{corollary}
	\label{cor_2Sing}
	Let $(g,\nabla)$ be a special K\"ahler structure on $\P^1$ such that the associated cubic form $\Xi$ is non-trivial and meromorphic.
	Then $(g, \nabla)$ must have at least three singularities.
\end{corollary}
\begin{proof}

Combining \eqref{Eq_NecessCondMixedSing} and the trivial inequality
$\sum_{j=1}^k\, \beta_j\leq \ell+\sum_{j=1}^k\,[\beta_j]$, we obtain
\[4(\gamma-1)<\sum_{j=1}^n\,\beta_j\leq 6(\gamma-1)+k.\]
Since  $\gamma=0$ for $\P^1$, we conclude $-4<-6+k$, i.e. $k\geq 3$.
\end{proof}

\begin{remark}
	For any $n\in \Z$ the metric $g = r^n |dz|^2$ is flat on $\C\setminus\{ 0 \}$, hence, can be thought of as a special K\"ahler structure on $\P^1$ singular at most at 2 points, namely $0$ and $\infty$.
	Notice that the corresponding cubic form is trivial, hence this example does not contradict \autoref{cor_2Sing}.
\end{remark}

\begin{example}
	Let $\cR_{24}^0$ denote the moduli space of all special K\"ahler metrics with  $24$ singular points all of logarithmic type.
	$\cR_{24}^0$ fibers over $\Sym^{24}(\P^1)\setminus \{\rm Diagonal\ subset\}$, where each fiber is homeomorphic to a Zariski open subset of $\CP^{18}$.
	Hence, $\cM_{24}^0$ has complex dimension $42$.
	If we also mod out by the natural action of $\PGL(2,\C)$, the resulting space is of complex dimension $39$.
	This space is of interest for elliptic K3 surfaces~\cite{GrossWilson00_LargeCxStrLimits}.
\end{example}

In what follows below we would like to describe which metrics actually appear as associated hyperbolic metrics of some special K\"ahler structure from   $\cM_k^\ell (\fp, \fb)$.
Thus, let $\cR_{\mathrm{hyp}}$ be the set of all hyperbolic metrics on $\Sigma$ with isolated singularities.
We have a natural map
\begin{equation*}
	T\colon \cM_k^\ell (\fp, \fb)\to \cR_{\mathrm{hyp}},\qquad T(g,\nabla)= \tilde g,
\end{equation*}
where $\tilde g$ is the hyperbolic metric associated with $(g,\nabla)$.

\begin{proposition}
		For any compact Riemann surface the image of $T$ is isomorphic to $H(\fp,\fb)/{\Bbb C}^*$.
		 Moreover, if $\Sigma=\P^1$, the image of $T$ consists of those hyperbolic metrics $\tilde g$, which satisfy the following: There exist  $r\ge 0$ points $q_1,\dots, q_r\in \P^1\setminus \{ p_1,\dots, p_k \}$ as well as $m\in \Z^k$ and $n\in \Z_{>0}^{r}$ such that the following holds:
		 	\begin{enumerate}[(a),itemsep=0pt,topsep=2pt]
		 		\item \label{It_HypMetrSmoothAwayFrom}
		 		$\tilde g$ is smooth on $\P^1\setminus \{ p_1,\dots, p_k , q_1,\dots, q_r\}$.
		 		\item For all $j\le \ell$ we have $m_j>\b_j -1$ and
			 		\begin{equation}
				 		\label{Eq_NumCondOrders}
			 			\sum_{j=1}^\ell m_j +\sum_{j=\ell}^k (\b_j-1) +\sum_{j=1}^r n_j = -6.
			 		\end{equation}
		 		\item \label{It_OrderConicalPt}
		 		If $j\le \ell$, then $\ord_{p_j}\tilde g = m_j -\b_j$.
		 		\item If $j>\ell$, then $\tilde g$ has a cusp singularity at $p_j$.
		 		\item \label{It_SpecialConicalPts}
		 		Each $q_j$ is a conical singularity of $\tilde g$ of order $n_j$.
		 	\end{enumerate}
\end{proposition}
\begin{proof}
	Using \autoref{NeceSuffCondition_spK}, consider $T$ as a map $T\colon H(\fp, \fb)/\R_{>0}\to \cR_{\mathrm{hyp}}$.
	 By \autoref{Prop_spKwithFixedMetric}, this yields an injective map $H(\fp, \fb)/\mathbb C^*\to \cR_{\mathrm{hyp}}$ with the same image.
	 Hence, we obtain that the image of $T$ is isomorphic to  $H(\fp,\fb)/{\Bbb C}^*$.
	
	Furthermore, think of $\P^1$ as the affine complex line $\C$ compactified by a point at infinity.
	Without loss of generality we can assume that none of  $p_j$ equals $\infty$.
	Then  for $\Xi = \Xi_0\, dz^3\in H(\fp, \fb)$ we have the following expression:
	\begin{equation}
		\label{Eq_AuxXi0}
		\Xi_0 =  A(z-p_1)^{m_1}\dots (z-p_\ell)^{m_\ell}(z-p_{\ell+1})^{\beta_{\ell+1}-1}\cdots(z-p_{k})^{\beta_{k}-1} (z-q_1)^{n_1}\dots (z-q_r)^{n_r}.
	\end{equation}
	Here $A\neq 0$ is a constant, $m_j$ is an integer satisfying $m_j>\beta_j-1$, $n_j$ is a positive integer, and
	$q_1,\dots, q_r$ are those zeros of $\Xi$, which are not contained in $\{ p_1,\dots, p_k \}$.
	Moreover, \eqref{Eq_NumCondOrders} holds since $\Xi$ is regular at $\infty$.
	
	If $[g,\nabla]\in \cM_k^\ell (\fp,\fb)$, where $g=e^{-u}|dz|^2$, then by~\eqref{Eq_UVrelation} we obtain that $\tilde g=e^{2v}|dz|^2$ is smooth on $\P^1\setminus \{ p_1,\dots, p_k , q_1,\dots, q_r\}$.
	Moreover,  \autoref{Thm_spKviaHyperbXi_general} yields \textit{\ref{It_OrderConicalPt}}--\textit{\ref{It_SpecialConicalPts}}.
	
	Conversely, given a hyperbolic metric $\tilde g$ satisfying \textit{\ref{It_HypMetrSmoothAwayFrom}}--\textit{\ref{It_SpecialConicalPts}}, define $\Xi_0$ by~\eqref{Eq_AuxXi0} and put $\Xi:=\Xi_0\, dz^3$.
	Then \autoref{Thm_spKviaHyperbXi_general} yields a special K\"ahler structure $[g,\nabla]\in \cM_{k}^\ell(\fp, \fb)$ whose associated hyperbolic metric is $\tilde g$.
\end{proof}

We note in passing that it is possible to define topologies, or even smooth structures, on $\cM_k^\ell (\fp, \fb)$ and $\cR_k^\ell (\fp, \fb)$ directly along the lines of~\cite{MazzeoWeiss17_TeichmThry}.
This would then require to prove that the map $\Xi$ is a homeomorphism, which seems to be excessive for our modest aims.

\section{Existence of special Kahler metrics on
Riemann surfaces with positive genera}

\begin{corollary}$\phantom{A}$
	\label{Cor_NScondition_spK_genus>0} Let $\Sigma$ be a compact Riemann surface with genus $\gamma>0$. Then $\cM_k^\ell (\fp, \fb)$ is non-empty if and only if  the following three conditions hold{\rm:}

\begin{enumerate}[(i),topsep=3pt,itemsep=1pt]
\item $\b_1+\cdots+\b_k>4\gamma-4$,
\item the line bundle $L=L(\fp,\,\fb)$ has a non-trivial holomorphic section, and
\item for all $\ell+1\leq j\leq k$ we have\ \; $\dim_{\Bbb C}\, H^0(L)>
\dim_{\Bbb C}\, H^0\big(L-p_j\big)$.
\end{enumerate}
Moreover, under the above conditions, $\cM_k^\ell (\fp, \fb)$  is homeomorphic to an open dense subset of a sphere of dimension $2N+1$, where
$N:=\dim_{\Bbb C}\, H^0(L)-1$.
The space $\cR_k^\ell (\fp, \fb)$ is homeomorphic to a Zariski open subset of $\CP^N$.
\end{corollary}
\begin{proof} Suppose that these three conditions hold.
By the last two ones, there exists a meromorphic cubic differential $\Xi$ which is holomorphic outside $\{p_1,\cdots, p_k\}$ and satisfies    $\ord_{p_j}\Xi \ge [\b_j]$ for $1\leq j\leq \ell$ and  $\ord_{p_j}\Xi = \b_j-1$ for  $\ell+1\leq j\leq k$.
The conclusion follows from the first condition and \autoref{NeceSuffCondition_spK}.

Suppose that $\cM_k^\ell (\fp, \fb)$ is non-empty. The first two conditions follows from Theorem \ref{NeceSuffCondition_spK}. We show the third one by contradiction. Suppose that there exist $\ell+1\leq j\leq k$ such that \[\dim_{\Bbb C}\, H^0(L)\leq
\dim_{\Bbb C}\, H^0\big(L-p_j\big).\]
Then we find $H^0(L-p_j)=H^0(L)$. We pick a special K\"ahler structure $(g,\,\nabla)$ in $\cM_k^\ell (\fp, \fb)$ with associated cubic differential $\Xi$. By Theorem \ref{NeceSuffCondition_spK}, we know that $\Xi$ has order $\b_j-1$ at $p_j$. On the other hand, since $\Xi$ belongs to $H^0(L)=H^0(L-p_j)$, the order of $\Xi$ at $p_j$ should be greater than or equal $\b_j$.
This is a contradiction.
\end{proof}

In the case $\Sigma$ is an elliptic curve (compact Riemann surface of genus one), the statement of the above corollary can be made more explicit.

\begin{corollary}$\phantom{B}$
	\label{Cor_NScondition_EllCurves} For an  elliptic curve  $E$ the space $\cM_k^\ell (\fp, \fb)$ is non-empty if and only if the following three conditions hold{\rm :}
	\begin{enumerate}[(i),topsep=3pt,itemsep=1pt]
	\item \label{It_EllCurveCond1}$\b_1+\cdots+\b_k>0$;
	\item The line bundle $L = L(\fp, \fb)$ has a non-trivial holomorphic section;
	\item\label{It_EllCurveCond3} If $\deg\, L=1$, then for all $\ell+1\leq j\leq k$, the divisor $-D(\fp, \fb) +p_j$ is \emph{not} equivalent to zero.
\end{enumerate}
Moreover, under these conditions, $\cM_k^\ell (\fp, \fb)$ is homeomorphic to an open dense subset of a sphere of dimension $2N+1$, where
$N:=\dim_{\Bbb C}\, H^0(L)-1$.
The space $\cR_k^\ell (\fp, \fb)$ is homeomorphic to a Zariski open subset of $\CP^N$.
\end{corollary}

\begin{proof}
	While a proof of this corollary could be obtained from \autoref{Cor_NScondition_spK_genus>0}, we prefer a more direct approach.
	
Thus, suppose there exists a special K\"ahler structure $(g,\,\nabla)$ as in the statement of this corollary.
 The first two condtions follow from Theorem \ref{NeceSuffCondition_spK} directly. Suppose $\deg\,L=1$. Then $H^0(L)$ has dimension one by the Riemann-Roch theorem.
 If there exists $\ell+1\leq j\leq k$ such that the divisor $-D+p_j:=-D(\fp,\,\fb) +p_j$ is equivalent to zero, then there exists an elliptic function $f$ such that $(f)= -D+p_j$ and $H^0(L)$ is generated by $f\,dz^3$, where $dz$ is a nonwhere vanishing holomorphic one-form on $E$.
 Furthermore, the associated cubic differential $\Xi$ equals ${\rm Const.}\,f\,dz^3$ and has order $\b_j$ at $p_j$. This is a contradiction, which finishes the proof of the ``only if'' part.
\medskip

Assume that \textit{\ref{It_EllCurveCond1}}--\textit{\ref{It_EllCurveCond3}} hold. We divide the proof of the ``if'' part into the following two cases.

\noindent {\it Case 1.}\ Suppose $\deg\, L=0$.
Then $L$ is trivial, in particular $L$ has a non-trivial holomorphic section, and there exists an elliptic function $f$ on $E$ such that $(f)=-D$.
By Theorem \ref{NeceSuffCondition_spK}, there exists a special K\"ahler structure $(g,\nabla)$ whose associated cubic differential is $f\,dz^3$.

\noindent {\it Case 2.}\  Suppose $d:=\deg\, L>0$. By the Abel-Jacobi theorem, we can find a point $q\in E={\Bbb C}/\Lambda$ and an elliptic function $f$ on $E$ such that $(f)=dq -D$.
If $q\notin\{p_{\ell+1},\cdots, p_k\}$, then by Theorem \ref{NeceSuffCondition_spK} there exists such a special K\"ahler structure $(g,\nabla)$ whose associated cubic differential is $f\,dz^3$, where $dz$ is a non-where vanishing holomorphic 1-form on $E$.
Hence, it remains to consider the case $q\in \{p_{\ell+1},\cdots, p_k\}$.

{\it Subcase 2.1.} Suppose $d\geq 2$.
Since there exist $q_1,\cdots,q_d$ in
$E\setminus \{p_{\ell+1},\cdots, p_k\}$ such that $q_1+\cdots+q_d\equiv dq$ ({\rm mod}\,$\Lambda$),
we can find an elliptic function $f$ on $E$ such that
$(f)=q_1+\cdots+q_d-D$.
We are done.

{\it Subcase 2.2.} Suppose $d=1$. Then there exists $\ell+1\leq j\leq k$ such that
$-D+p_j\sim 0$.
However, this possibility is excluded by~\textit{\ref{It_EllCurveCond3}}.
\end{proof}

		

We conclude this section by proposing a conjecture about the existence of certain cubic differentials and special K\" ahler structures on compact Riemann surfaces of genera greater than one.
\begin{conj}
\label{conj:spk} {\rm
Let $\gamma$ be an integer greater than $1$ and ${\frak b}=(\beta_1,\cdots,\beta_k)$ be as in Definition \ref{defn:moduli}. Assume that the numerical condition \eqref{Eq_NecessCondMixedSing} holds. Then there exists a compact Riemann surface $\Sigma$ of genus $\gamma$ and a $k$-tuple ${\frak p}=(p_1,\cdots, p_k)$ of pairwise distinct points in $\Sigma$ such that $H(\fp,\fb)$---and consequently also $\cM_k^\ell (\fp, \fb)$---is non-empty.}
\end{conj}

\section{On compactifications of the moduli spaces}

A consequence of \autoref{Cor_NecessCondition_spK}	is that the moduli spaces $\cM_k^\ell (\fp, \fb)$ and $\cR_k^\ell (\fp, \fb)$ are  non-compact provided there is at least one logarithmic singularity, i.e., if $\ell <k$.
The purpose of this section is to describe compactifications of these moduli spaces.

\medskip

Notice first, that we have the following natural stratification
\begin{equation}
	\label{Eq_StratifCubicDifferentials}
	\rH^0(L)\setminus \{ 0 \} = \bigcup_{\fm\in \Z_{\ge 0}^{k-\ell}} H(\fp, \fb_\fm),
\end{equation}
where $\fb_\fm:= \fb + (0,\fm)$.
Clearly, $H(\fp, \fb)=H(\fp, \fb_0)$ is the stratum with the highest dimension.
For example, in the case $\Sigma = \P^1$ we have
\begin{equation*}
	\dim_\C H(\fp,\fb_m) =-6+k-\ell -\sum_{j=1}^k [\b_j] - \sum_{j=1}^{k-\ell} m_j
\end{equation*}
provided $H(\fp,\fb_m)$ is non-empty.

Moreover, this stratification is in fact finite.
Indeed, if $\gamma$ is the genus of $\Sigma$ and $\Xi\in H(\fp, \fb_\fm)$, then
\begin{equation*}
	\begin{aligned}
		6(\gamma -1)&=\sum_{p\in \supp(\Xi)} \ord_p\Xi\  \ge   \sum_{j=1}^k  \ord_{p_j}\Xi
		\ge
		\sum_{ j\le\ell } (\beta_j -1) + \sum_{ j>\ell} (\beta_j + m_{j-\ell} -1)\\
		& = |\fb| + |\fm| -k,
	\end{aligned}
\end{equation*}
which implies the upper bound $|\fm|\le k-6(\gamma -1) -|\fb|$.

Notice also that each stratum is invariant under the action of $\C^*$.

\medskip

Starting from a different perspective, consider the set $\overline\cM_k^\ell (\fp, \fb)$ which consists of all special K\"ahler structures on $\Sigma$ satisfying the following:
\begin{itemize}[itemsep=0pt,topsep=3pt]
	\item The associated cubic form $\Xi$ is non-trivial and meromorphic;
	\item For $j\le \ell$  we have a conical singularity at $p_j$ and $\ord_{p_j} (g,\nabla) =\tfrac 12 \b_j$;
	\item  For $j>\ell$  we have a logarithmic singularity at  $p_j$ and $\ord_{p_j} (g,\nabla)\ge \frac 12 \b_j$.
\end{itemize}
Just like above, we do not distinguish two special K\"ahler structures if they differ by a rescaling of the metric.

Notice that the only difference in the definitions of $\overline\cM_k^\ell (\fp, \fb)$ and $\cM_k^\ell (\fp, \fb)$ is that $\frac 12\b_j$ is only \emph{a lower bound} for the order of the logarithmic singularity at $p_j$ ($j> \ell$).

Furthermore,  $\overline\cM_k^\ell (\fp, \fb)$ admits a stratification akin to~\eqref{Eq_StratifCubicDifferentials}, namely
\begin{equation}
	\label{Eq_StratifModulSpace}
	\overline\cM_k^\ell (\fp, \fb) =\bigcup_{\fm\in \Z_{\ge 0}^{k-\ell}} \cM_k^\ell (\fp, \fb_\fm)
\end{equation}
with the top stratum being $\cM_k^\ell (\fp, \fb)$.
Moreover, the map $\Xi$ extends as a bijective  map
\begin{equation*}
	\bar\Xi\colon \overline\cM_k^\ell (\fp, \fb)\to \rH^0(L)\setminus\{ 0\}/\R_{>0}
\end{equation*}
such that each $\cM_k^\ell (\fp, \fb_\fm)$ is mapped bijectively to  $H(\fp, \fb_\fm)$.
In particular, \eqref{Eq_StratifModulSpace} consists of a finite number of strata.
Just like in the case of $\cM_k^\ell(\fp,\fb)$, we use $\bar\Xi$ to endow $\overline\cM_k^\ell (\fp, \fb)$ with a topology.

Clearly, we can also construct a compactification of $\cR_k^\ell (\fp, \fb)$ in a similar manner.
Namely, defining
\begin{equation*}
	\overline\cR_k^\ell (\fp, \fb):=\bigcup_{\fm\in \Z_{\ge 0}^{k-\ell}} \cR_k^\ell (\fp, \fb_\fm)
\end{equation*}
we obtain  a bijection
\begin{equation*}
	\bar \xi\colon \overline\cR_k^\ell (\fp, \fb)\to \rH^0(L)\setminus \{ 0 \}/\C^* = \P\rH^0(L)
\end{equation*}
fitting into the commutative diagram
\begin{equation*}
	\begin{CD}
		\overline\cM_k^\ell (\fp, \fb)@>\bar\Xi>>S(\rH^0(L))\\
		@V\pi VV @VV V\\
		\overline\cR_k^\ell (\fp, \fb) @>\bar\xi >> \P\rH^0(L),
\end{CD}
\end{equation*}
cf.~\eqref{Diagr_Moduli}.
Here $S(\rH^0(L))$ denotes the sphere of unite radius in $H^0(L)$ with respect to some norm.

Summarizing, we obtain the following.
\begin{proposition}
\label{prop:cptf}
	The map $\bar \Xi$ establishes a natural bijective correspondence between $\overline\cM_k^\ell (\fp, \fb)$ and the unit sphere in $\rH^0(L)$.
	Likely, $\bar \xi$  establishes a natural bijective correspondence between $\overline\cR_k^\ell (\fp, \fb)$ and  $\P\rH^0(L)$. \qed
\end{proposition}

\bibliography{references}
\end{document}